\documentclass[11pt,letterpaper,reqno]{amsart}
\usepackage{graphicx}
\usepackage{amsmath}
\usepackage{amssymb}
\usepackage{amsfonts}
\usepackage{amsrefs}
\usepackage{amsthm}
\usepackage{cancel}
\usepackage{enumerate}
\usepackage[mathscr]{eucal}

\newcommand{\R}{\mathbb{R}}

\newcommand{\inv}{^{-1}}
\newcommand{\sm}{\setminus}

\newtheorem{thm}{Theorem}
\newtheorem{lemma}[thm]{Lemma}

\newtheorem{observation}[thm]{Observation}

\theoremstyle{definition}
\newtheorem{definition}[thm]{Definition}

\theoremstyle{remark}
\newtheorem{remark}{Remark}
\newtheorem*{ack}{Acknowledgments}
\newtheorem*{outline}{Outline}

\DeclareMathOperator{\vol}{vol}
\DeclareMathOperator{\graph}{graph}
\DeclareMathOperator{\dist}{dist}
\DeclareMathOperator{\interior}{interior}

\DeclareMathOperator{\RotSym}{RotSym_{\emph{n}}}
\DeclareMathOperator{\RotSymm}{RotSym}

\begin{document}
\title{On the lower semicontinuity of the ADM mass}
\date{\today}
\author{Jeffrey L. Jauregui}
\address{Dept. of Mathematics,
Union College,
Schenectady, NY 12308}
\email{jaureguj@union.edu}

\begin{abstract}
The ADM mass, viewed as a functional on the space of asymptotically flat Riemannian metrics of nonnegative scalar curvature, fails to be continuous for many natural topologies.  In this paper we prove that lower semicontinuity holds in natural settings: first, for pointed Cheeger--Gromov convergence (without any symmetry assumptions) for $n=3$, and second, assuming rotational symmetry, for weak convergence of the associated canonical embeddings into Euclidean space, for $n \geq 3$.  We also apply recent results of LeFloch and Sormani to deal with the rotationally symmetric case, with respect to a pointed type of intrinsic flat convergence.  We provide several examples, one of which demonstrates that the positive mass theorem is implied by a statement of the lower semicontinuity of the ADM mass.
\end{abstract}

\maketitle

\section{Introduction}
In general relativity a number of important open questions involve taking a limit of a sequence $\{(M_i,g_i)\}_{i=1}^\infty$ of asymptotically flat Riemannian manifolds of nonnegative scalar curvature. For instance, such limits arise in stability problems (e.g. \cites{bray_finster,finster_kath, lee,lee_sormani,lee_sormani2,lefloch_sormani,huang_lee, huang_lee_sormani}) and in flows of asymptotically flat manifolds (e.g. \cites{dai_ma, haslhofer, list, OW}).  In these contexts it is desirable to understand how the ADM (total) masses \cite{adm} of $(M_i,g_i)$ compare to the ADM mass of the limit space, $(N,h)$.  Recall that an asymptotically flat manifold can be viewed as an initial data set for the Einstein equations in general relativity, and the ADM mass represents the total mass contained therein.

While it is well-known that the ADM mass is not continuous with respect to many natural topologies, some examples (see section \ref{sec_examples}) suggest that lower semicontinuity ought to hold:
\begin{equation}
\label{eqn_lsc}
m_{ADM}(N,h) \leq \liminf_{i \to \infty} m_{ADM}(M_i,g_i).
\end{equation}
In this paper the first main result is a proof of (\ref{eqn_lsc}) for pointed Cheeger--Gromov convergence (see Definition \ref{def_CG}), subject to natural hypotheses:
\begin{thm}
\label{thm_CG}
Let $\{(M_i, g_i,p_i)\}$ be a sequence of asymptotically flat pointed Riemannian $3$-manifolds that converges in the pointed $C^2$ Cheeger--Gromov sense to an asymptotically flat pointed Riemannian $3$-manifold $(N,h,q)$.    Assume that for each $i$, $(M_i,g_i)$ contains no compact minimal surfaces and that $g_i$ has nonnegative scalar curvature.  Then (\ref{eqn_lsc}) holds.
\end{thm}

The second main result is a proof of (\ref{eqn_lsc}) for rotationally symmetric asymptotically flat manifolds, subject to similar hypotheses.  The topology we consider here is that of weak convergence (in the sense of currents) of canonical isometric embeddings into Euclidean space\footnote{It is an unfortunate coincidence of terminology that the ADM mass is unrelated to the \emph{mass} of rectifiable currents, which is well-known to be lower semicontinuous with respect to weak convergence.}.  A more formal statement is given as Theorem \ref{thm_main} in section \ref{sec_rot_sym}.
\begin{thm}[Informal statement]
\label{thm_intro}
Let $\{(M_i,g_i)\}$ denote a sequence of rotationally symmetric, asymptotically flat Riemannian manifolds of nonnegative scalar curvature.  Assume $\partial M_i$ is either empty or a minimal surface, and that $M_i$ contains no other compact minimal surfaces.  If $(M_i, g_i)$ converges weakly to some $(N,h)$ then (\ref{eqn_lsc}) holds.
\end{thm}

The first source of motivation we present arises from a conjecture of Bartnik pertaining to the quasi-local mass problem in general relativity \cites{bartnik,bartnik2}.  Suppose $\Omega$
is a compact Riemannian 3-manifold with boundary $\partial\Omega$.  Assume $\Omega$ has nonnegative scalar curvature, and let $\partial\Omega$ have mean curvature $H$ and induced metric $\gamma$.  Consider all asymptotically flat 3-manifolds $(N,h)$
that have nonnegative scalar curvature, contain no compact minimal surfaces, and whose boundary is compact and isometric to $\gamma$ with corresponding mean curvature $H$.  
(The significance of matching the boundary metrics and mean curvatures is that nonnegative scalar curvature holds in a distributional sense across the interface.) 
The infimum of the ADM mass among all such $(N,h)$ is called the \emph{Bartnik mass} of $\Omega$ and is a well-known example of a quasi-local mass functional. 

Bartnik conjectured that the above infimum is  achieved by some $(N,h)$, called a \emph{minimal mass extension}.  One program to approach this problem directly is to consider
an ADM-mass-minimizing sequence $(M_i, g_i)$ and attempt to extract a convergent subsequence in some topology, say with limit $(N,h)$.  However, finding such a topology, together with a compactness theorem, remains a highly nontrivial open problem.  An additional step necessary to show $(N,h)$ is in fact a minimal mass extension would be that the ADM mass does not increase when passing to the limit, i.e., that (\ref{eqn_lsc}) holds.

A second source of motivation lies in the goal finding a new proof of the positive mass theorem \cites{schoen_yau,witten} that uses Ricci flow or a related flow (see, e.g., \cites{OW,dai_ma,list,haslhofer}).  A sketch of a proof would be to initiate Ricci flow on an asymptotically flat manifold $(M,g)$ of nonnegative scalar curvature.  Short-time existence is known, asymptotic flatness and the nonnegativity of scalar curvature are preserved, and the ADM mass is constant along the flow \cites{dai_ma,OW}.
Performing surgery as needed, one might attempt to show that the space eventually converges to Euclidean space in some sense.  However, since the ADM mass is constant under the Ricci flow, it is clear that the convergence cannot be  in a topology for which the ADM mass is continuous.  To prove nonnegativity of the ADM mass of the initial space, it would be necessary to know the ADM mass does not increase when passing to the limit.

\begin{outline}
In section \ref{sec_examples} we give a number of examples to illustrate some of the subtleties of the problem, to motivate why lower semicontinuity is plausible, and to demonstrate why the hypotheses in Theorem \ref{thm_CG} are necessary.  Example 3 is of particular interest, because it gives a sense in which lower semicontinuity implies the positive mass theorem.  Section \ref{sec_CG} contains a proof of Theorem \ref{thm_CG}; section \ref{sec_rot_sym} provides some preliminaries for rotationally symmetric manifolds before giving the proof of Theorem \ref{thm_main}. We conclude in section \ref{sec_IF} with a proof of lower semicontinuity in rotational symmetry for a type of pointed intrinsic flat convergence \cite{sormani_wenger}, using recent results of LeFloch and Sormani \cite{lefloch_sormani}.
\end{outline}

\begin{ack}
The author would like to thank Graham Cox and Christina Sormani for helpful discussions.
\end{ack}

\section{Examples}
\label{sec_examples}
We give a series of examples to illustrate some of the issues present in dealing with convergence of asymptotically flat manifolds and the behavior of the ADM mass.
In particular, examples 1 and 4 below show that the ADM mass is \emph{not} lower semi-continuous with respect to any reasonable 
notion of pointed convergence, without some additional assumptions on the scalar curvature and the absence of compact minimal surfaces.  

Recall that pointed notions of convergence are natural to consider for noncompact manifolds.  By ``a reasonable notion of pointed convergence'' of pointed Riemannian manifolds $(M_i,g_i,p_i)$ to $(N,h,q)$, we mean \emph{any} type of convergence that satisfies the following property:  if given any $r>0$, the metric ball of radius $r$ about $p_i$ in $(M_i,g_i)$ is isometric to such a ball about $q$ in $(N,h)$, for $i$ sufficiently large, then $(M_i, g_i, p_i)$ converges to $(N,h,q)$.

\vspace{2mm}
\paragraph{\emph{Example 1 (Flattened-out Schwarzschild):}}  Fix a constant $m>0$.  
Let $(M,g)$ be the graph in $\R^4$, endowed with the induced metric, of the function $f:\R^3 \sm B_{2m}(\vec 0)\to \R$ given by $f(x) = \sqrt{8m(|x|-2m)}$, where $|x|$ is the distance to the origin.  $(M,g)$ is
isometric to the (spatial) Schwarz\-schild manifold of ADM mass $m$. This graph, represented in one lower dimension, is depicted in figure \ref{fig_schwarz}.  The minimal surface $\partial M$ is called the horizon of $M$.

\begin{figure}
\begin{center}
\includegraphics[height=1.3in]{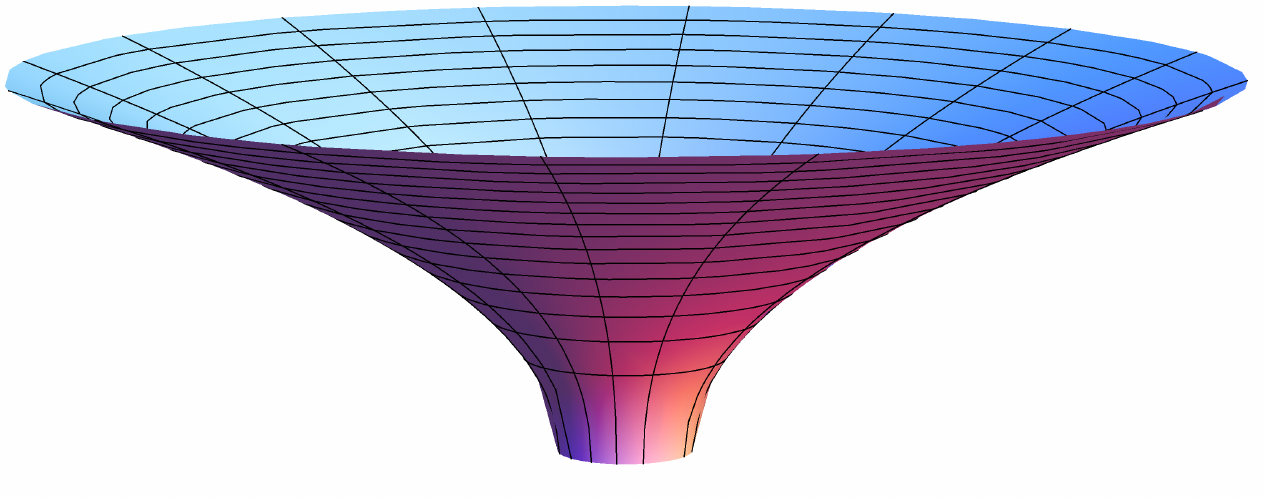}
\caption{\small The Schwarzschild manifold represented as a graph. \label{fig_schwarz}}
\end{center}
\end{figure}

For each integer $i > 0$, define
$$f_i(x) = \min(i, f(x)),$$
whose graph is depicted in figure \ref{fig_schwarz_out}.  Now, smooth $f_i$ on a small annulus about its non-smooth set, and call the result $\tilde f_i$.  Let $(M_i, g_i)$ be the graph of $\tilde f_i$, which is an asymptotically flat manifold.  In fact, the ADM mass of $(M_i,g_i)$ vanishes for each $i$, since outside a compact set, $\tilde f_i$ is constant (and so $g_i$ is flat). 
Let $p = (2m,0,0,0)$, which belongs to $M$ and all $M_i$.  Then $(M_i, g_i,p)$ converges in any reasonable pointed sense to $(M,g,p)$.  However, 
$$\liminf_{i \to \infty} m_{ADM}(M_i,g_i) = 0 < m = m_{ADM} (M,g),$$
violating (\ref{eqn_lsc}).

Note that in this example, each $(M_i,g_i)$ has negative scalar curvature somewhere, no matter how the smoothing is performed.  (This follows from the equality case of the positive mass theorem \cites{schoen_yau,witten}.)  This example reveals that to establish lower semicontinuity, nonnegative scalar curvature is a necessary hypothesis.

\begin{figure}
\begin{center}
\includegraphics[height=1.3in]{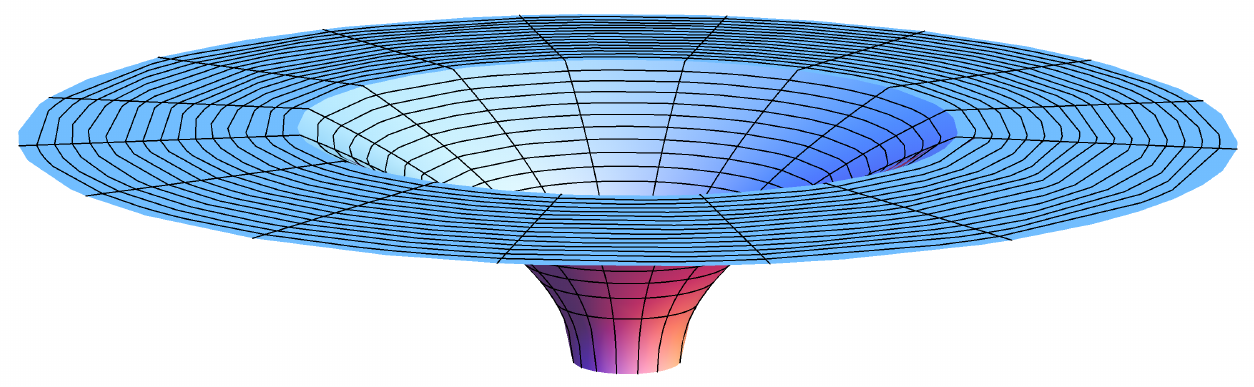}
\caption{\small An illustration of the ``flattened-out'' Schwarzschild manifold of Example 1. \label{fig_schwarz_out}}
\end{center}
\end{figure}

\vspace{2mm}
\paragraph{\emph{Example 2 (Flattened-in Schwarzschild):}}
Proceed as in the previous example, but with an alternative definition:
$$f_i(x) = \max(i, f(x)),$$
where $f(x)$ has been extended by zero to $\R^3$.  The graph of $f_i$ is shown in figure \ref{fig_schwarz_in}. 
Upon a suitable smoothing to $\tilde f_i$, the graph of $\tilde f_i$, call it $(M_i,g_i)$, can be made to have nonnegative scalar curvature.  Moreover, for $p_i=(0,0,0,i)$, 
$(M_i,g_i,p_i)$ converges in any reasonable pointed sense to Euclidean space $(\R^3,\delta_{ij},\vec 0)$.  Moreover, each $(M_i,g_i)$ has ADM mass $m$.  This example shows that the ADM mass can indeed drop when passing to a limit.

To add a different twist to this example, scale the metric $g_i$ by a constant $c^2_i>0$, where $c_i \nearrow \infty$.  The corresponding ADM masses are scaled by $c_i$.
Then $(M_i,c_i^2 g_i,p_i)$ still converges to Euclidean space.  Moreover,
$\liminf_{i \to \infty} (M_i,c_i^2 g_i) = +\infty$, while the limit has zero mass.

\begin{figure}
\begin{center}
\includegraphics[height=1.1in]{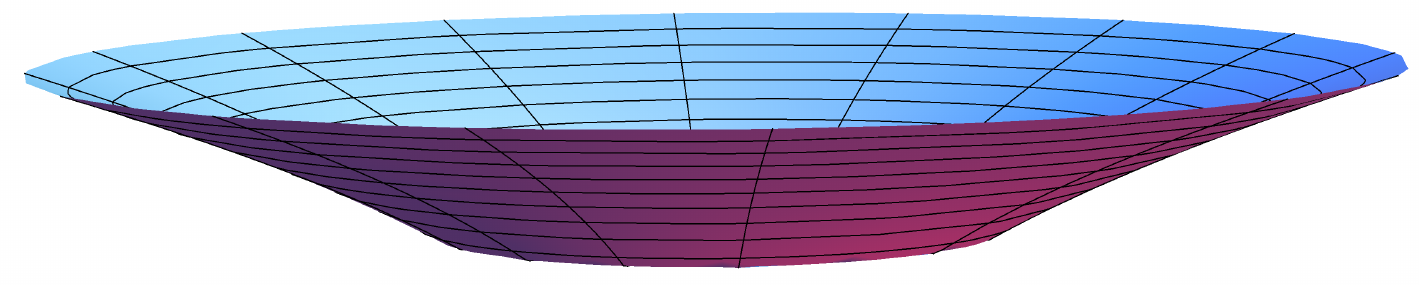}
\caption{\small An illustration of the ``flattened-in'' Schwarzschild manifold of Example 2. \label{fig_schwarz_in}}
\end{center}
\end{figure}

\vspace{2mm}
\paragraph{\emph{Example 3 (Blowing up a fixed manifold):}}
Let $(M,g)$ be an asymptotically flat $n$-manifold of nonnegative scalar curvature, with ADM mass $m \in \R$. For each positive integer $i$, let $g_i$ denote the rescaled metric $i^2 g$.  Each $g_i$ is asymptotically flat with nonnegative scalar curvature, with ADM mass $i^{n-2}m$.  In particular, the possible values of $\liminf_{i \to \infty} m_{ADM}(M,g_i)$ are $-\infty, 0,$ or $+\infty$ according to the sign of $m$.

Fix a point $p \in M$.  Using the exponential map of $g$ about $p$, composed with a scaling by $i\inv$, and the smoothness of $g$, it is straightforward to show that $(M,g_i,p)$ converges in the pointed $C^2$ Cheeger--Gromov sense (see Definition \ref{def_CG}) to Euclidean space $(\R^n,\delta,\vec 0)$. Thus, the limit space has mass $0$.  A lower semicontinuity statement would imply that $m \geq 0$. In particular, we see:
\begin{observation}
\label{obs}
The statement ``the ADM mass is lower semicontinuous on the space of asymptotically flat $n$-manifolds of nonnegative scalar curvature (that contain no compact minimal hypersurfaces), with respect to pointed $C^2$ Cheeger--Gromov convergence'' implies the positive mass theorem in dimension $n$ for such manifolds.
\end{observation}
This illustrates the depth of lower semicontinuity and in particular restricts any possible proof to the tools used in a proof of the positive mass theorem itself.\footnote{Indeed, our proof of Theorem \ref{thm_CG} uses Huisken--Ilmanen's results on inverse mean curvature flow \cite{imcf}, which are well-known to give an independent proof of the positive mass theorem.}

\vspace{2mm}
\paragraph{\emph{Example 4 (Hidden regions):}}
Let $(M,g)$ be the Schwarzschild manifold of mass $m>0$, described in Example 1. Fix $r > 2m$, and let $\Omega_r$ be the closed region between the coordinate sphere $S_r$ and the horizon.

Recall that the Schwarzschild manifold can be reflected (doubled) across its horizon to produce a smooth manifold with two asymptotically flat ends.  For a fixed $\epsilon  \in (0,m)$, the doubled Schwarzschild manifold  of mass $\epsilon$ contains a unique coordinate sphere $\Sigma$ in the doubled end
of the same area as $S_r$.

By identifying $S_r$ and $\Sigma$ as in figure \ref{fig_hidden}, it is possible to glue $\Omega_r$ into the doubled Schwarzschild manifold of mass $\epsilon$ so that the metric is Lipschitz and the scalar curvature is distributionally nonnegative at the interface.  The resulting manifold ``hides'' $\Omega_r$ inside the horizon of the doubled Schwarzschild manifold of mass $\epsilon$.
\begin{figure}
\begin{center}
\includegraphics[height=2.1in]{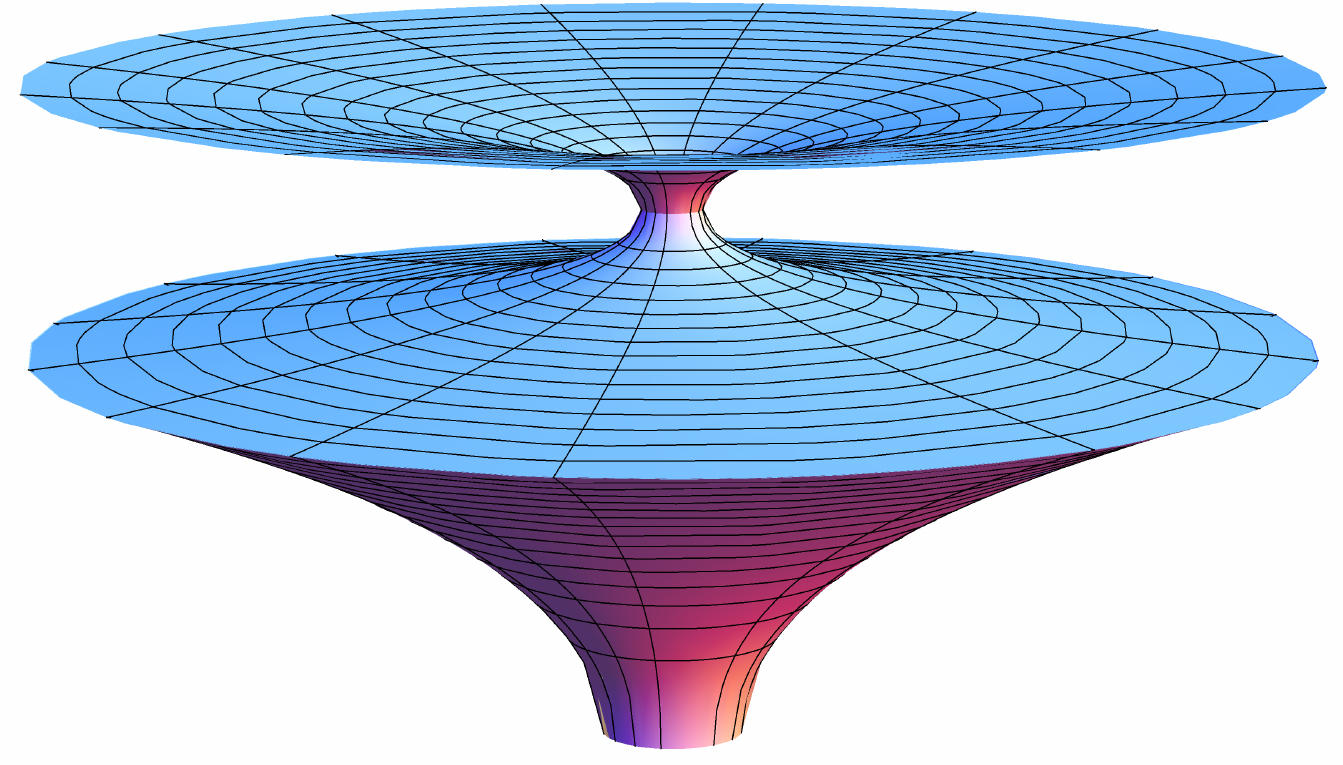}
\caption{\small An illustration of the manifold with a hidden region in Example 4. \label{fig_hidden}}
\end{center}
\end{figure}

By increasing $r$ to infinity (and keeping $\epsilon$ fixed but adjusting $\Sigma$ appropriately), we obtain a sequence of asymptotically flat manifolds of nonnegative scalar curvature each of which has mass $\epsilon$,  that converges in any reasonable pointed sense to the Schwarzschild manifold of mass $m> \epsilon$ (where the base points are chosen to be on the boundary).  Thus, the ADM mass jumps up in the limit.  This example illustrates the necessity of assuming the manifolds do not contain compact minimal surfaces in order to establish (\ref{eqn_lsc}).

\vspace{2mm}
\paragraph{\emph{Example 5 (Ricci flow of asymptotically flat manifolds):}} 
A given asymptotically flat $n$-manifold $(M,g_0)$ of nonnegative scalar curvature, $n \geq 3$, may be considered as initial data for the Ricci flow.  In the literature it has been established that there exists a solution to the Ricci flow $\{g_t\}$, $t \in [0,T)$, with initial condition $g_0$.  Asymptotic flatness and nonnegative scalar curvature are preserved under the flow, and interestingly, so is the ADM mass \cites{dai_ma,OW}.  For the class of rotationally symmetric manifolds $(M,g_0)$ of nonnegative scalar curvature and containing no compact minimal surfaces, Oliynyk and Woolgar showed that $\{g_t\}$ exists for all time ($T=+\infty$) and converges in the pointed $C^k$ Cheeger--Gromov sense to Euclidean space for every $k$ \cite{OW}.  In particular, if the initial metric has strictly positive ADM mass $m_0$, then all $g_t$ have ADM mass equal to $m_0$, and thus the mass jumps down to zero in the limit.

\vspace{2mm}
\paragraph{\emph{Example 6 (Badly behaved limit):}} Last we give an example of a sequence of asymptotically flat manifolds $M_i$ that converges in any reasonable pointed sense to a limit space that is not asymptotically flat.  In particular, the ADM mass of the limit is not even defined.  

A simple example is to append to the horizon of the Schwarzschild manifold a round cylinder $S^2 \times [0,L]$ of length $L$ and appropriate radius.  The resulting metric has $C^{1,1}$ regularity and nonnegative scalar curvature across the interface in an appropriate distributional sense.  With respect to any point on the boundary of the cylinder, letting $L \nearrow \infty$ produces a sequence of asymptotically flat manifolds that converges in any reasonable pointed sense to a half-infinite cylinder $S^2 \times [0,\infty)$, which is not asymptotically flat.

Even more extreme examples are possible, for instance by glueing small regions of nontrivial topology near a sequence of points that escapes to infinity.

\section{Lower semicontinuity for pointed Cheeger--Gromov convergence}
\label{sec_CG}
In this section we prove the lower semicontinuity of the ADM mass in dimension three with respect to pointed $C^2$ Cheeger--Gromov convergence.  The restriction $n=3$ arises primarily from the use of Huisken--Ilmanen's weakly-defined inverse mean curvature flow \cite{imcf}, and we conjecture that the analogous result holds for all dimensions $n\geq 3$.\footnote{This conjecture is not as benign as it may seen, because of Observation \ref{obs}: a proof of lower semicontinuity implies a proof of the positive mass theorem.  In higher dimensions, this remains an unresolved issue.}  Now we give some relevant definitions.

\begin{definition}
\label{def_CG}
A sequence of complete, pointed Riemannian $3$-manifolds $(M_i,g_i,p_i)$ converges in the \textbf{pointed $C^{k}$ Cheeger--Gromov sense} to a complete pointed Riemannian $3$-manifold $(N,h,q)$ if for every $r > 0$ there exists a domain $\Omega$ containing $B_r(q)$ in $(N,h)$, and there exist (for all $i$ sufficiently large) smooth embeddings $\Phi_i: \Omega \to M_i$ such that $\Phi_i(\Omega)$ contains the open $g_i$-ball of radius $r$ about $p_i$, and the metrics $\Phi_i^* g_i$ converge in $C^{k}$ to $h$ on $\Omega$.
\end{definition}
Note that no $M_i$ need be diffeomorphic to $N$ in the above definition.

\begin{definition}
\label{def_AF}
A smooth, connected, Riemannian 3-manifold $(M,g)$ (without boundary) is \textbf{asymptotically flat} (of order $\tau>\frac{1}{2}$) if 
\begin{enumerate}[(i)]
\item there exists a compact set $K \subset M$ and a diffeomorphism $\Phi: M \sm K \to \R^3 \sm B$ (where $B$ is a closed ball), and
\item in the coordinates $(x^1, x^2, x^3)$ on $M \sm K$ induced by $\Phi$, the metric obeys the decay conditions:
\begin{align*}
 \left|g_{jk} - \delta_{jk}\right| &\leq \frac{c}{|x|^\tau},&
 \left|\frac{\partial g_{jk}}{\partial x_l} \right| &\leq \frac{c}{|x|^{\tau+1}},\\
 \left|\frac{\partial^2g_{jk}}{\partial x_l \partial x_m} \right| &\leq \frac{c}{|x|^{\tau+2}},&
 |R| &\leq \frac{c}{|x|^q},
\end{align*}
for $|x|=\sqrt{(x^1)^2 + (x^2)^2 + (x^3)^2}$ sufficiently large and $j,k,l,m=1,2,3$, where $c>0$ and $q >3$ are constants,
$\delta_{ij}$ is the Kronecker delta, and $R$ is the scalar curvature of $g$.  
\end{enumerate}
Such $(x^i)$ form an \textbf{asymptotically flat coordinate system}.
\end{definition}

For such manifolds, the ADM mass \cite{adm} is well-defined \cite{bartnik_adm} by the limit
\begin{equation}
\label{eqn_ADM}
m_{ADM}(M,g) = \frac{1}{16\pi} \lim_{r \to \infty}  \sum_{i,j=1}^3\int_{S_r} \left(\frac{\partial g_{ij}}{\partial x^i}  - \frac{\partial g_{ii}}{\partial x^j} \right)\frac{x^j}{r} dA,
\end{equation}
where $(x^i)$ are asymptotically flat coordinates, $S_r$ is the coordinate sphere $\{|x|=r\}$, and $dA$ is the area form on $S_r$.

We restate Theorem \ref{thm_CG} here for the reader's convenience.
\begin{thm}
\label{thm_CG2}
Let $\{(M_i, g_i,p_i)\}$ be a sequence of asymptotically flat pointed Riemannian $3$-manifolds that converges in the pointed $C^2$ Cheeger--Gromov sense to an asymptotically flat pointed Riemannian $3$-manifold $(N,h,q)$.    Assume that for each $i$, $(M_i,g_i)$ contains no compact minimal surfaces and that $g_i$ has nonnegative scalar curvature.  Then
$$m_{ADM}(N,h) \leq \liminf m_{ADM}(M_i,g_i)  .$$
\end{thm}

In light of examples 1 and 4 of section \ref{sec_examples}, the last two hypotheses are necessary.

The key estimate in the proof is the following famous result of Huisken and Ilmanen \cite{imcf}, whose proof utilizes a weakly-defined inverse mean curvature flow along which the Hawking mass is non-decreasing.  Recall the Hawking mass of a hypersurface $\Sigma$ with area $A$, area form $dA$, and mean curvature $H$ in a Riemannian $n$-manifold is defined by the formula
\begin{equation}
\label{def_hawking}
m_H(\Sigma) = \frac{1}{2}\left(\frac{A}{\omega_{n-1}}\right)^{\frac{n-2}
{n-1}}
\left(1-\frac{1}{(n-1)^2}\left(\frac{1}{\omega_{n-1}} \int_\Sigma |H|^{n-1} dA\right)^{\frac{2}{n-1}}\right),
\end{equation}
where $\omega_{n-1}$ is the hypersurface area of the unit $(n-1)$-sphere in $\R^n$. The result below, while not stated explicitly in their paper, is well-known and follows from the theorems therein:

\begin{thm}[Huisken--Ilmanen \cite{imcf}]  \label{thm_imcf}
Let $(M,g)$ be an asymptotically flat 3-manifold of nonnegative scalar curvature, with connected nonempty boundary $\Sigma$.  If $\Sigma$ is \textbf{area-outer-minimizing} (i.e.,
every surface enclosing $\Sigma$ has area at least that of $\Sigma$), then
$$m_{ADM}(M,g) \geq m_H(\Sigma).$$
\end{thm}

\begin{proof}[Proof of Theorem \ref{thm_CG}]
Fix some parameter $\eta_0 \in (0, 0.01)$, which will be used for estimating the error in geometric quantities with respect to different metrics.  Let $\epsilon>0$. 
Fix an asymptotically flat coordinate system $(x^1, x^2, x^3)$ on $(N,h)$ defined for $r \geq r_0$, where $r=\sqrt{(x^1)^2 + (x^2)^2 + (x^3)^2}$.  Let $S_r$ be the coordinate sphere of radius $r$, and let $B_r$ denote the closure of the compact region in $N$ that is bounded by $S_r$.  Note that the coordinate system naturally induces a Euclidean metric $\delta$ on $N \sm B_{r_0}$.

By asymptotic flatness, we may increase $r_0$ if necessary to guarantee:
\begin{enumerate}[(i)]
\item $m_{ADM}(N,h) - m_H(S_r) \leq \frac{\epsilon}{2}$ for all $r \geq r_0$.\footnote{This is possible because $m_{ADM}(N,h) = \lim_{r \to \infty} m_H(S_r)$.  This equality is well-known and follows from results in \cite{FST}, for instance.}
\item $S_r$ has positive mean curvature in $(N,h)$ for all $r \geq r_0$.
\item Distances in $N \sm B_{r_0}$ with respect to $h$ and $\delta$ differ by a factor of at most $1+\eta_0$.
\item Areas of surfaces in $N \sm B_{r_0}$ with respect to $h$ and $\delta$ differ by a factor of at most $1+\eta_0$.
%\item $|h_{jk} - \delta_{jk}| \leq \delta$ outside $S_{r_0}$, where $h_{jk}$ is the expression of $h$ is the fixed coordinate system.
%\item $|\partial_l h_{jk}| \leq \delta$ outside $S_{r_0}$. 
%\item $|\partial_l\partial_m h_{jk}| \leq \delta$ outside $S_{r_0}$.
\end{enumerate}
From asymptotic flatness, the sectional curvature $\kappa$ of $(N,h)$ is of order $O(r^{-2-\tau})$.  In particular, there exists a constant $\kappa_0>0$ such that $|\kappa| \leq \kappa_0r^{-2-\tau}$ on $N\sm B_r$, for $r \geq r_0$.
Thus, we may increase $r_0$ if necessary so as to assume that 
\begin{enumerate}
%\item[(v)] $e^{2\sqrt{\kappa_0r_0^{-2-\tau}}r_0} < 1 + \eta_0$.
\item[(v)] $|\kappa|$ is bounded above on $N \sm B_{r_0}$ by some constant $\kappa_1$, where $e^{6\sqrt{\kappa_1} r_0} \leq 1 + \eta_0$.
\end{enumerate}
Moreover, by considering the rescaled manifolds $(N\sm B_r, r^{-2} h)$ for $r$ large (which converge in an appropriate $C^2$ sense to Euclidean space minus a ball), we see that one may choose $r_0$ large enough so that:
\begin{enumerate}
\item[(vi)] Any point in $S_{4r_0}$ has injectivity radius (with respect to $h$) at least $\frac{3r_0}{1+\eta_0}$.
\end{enumerate}

Choose $R>0$ large so that the ball of radius $R$ about $q$ in $(N,h)$ contains $B_{7r_0}$.  Use the definition of pointed Cheeger--Gromov convergence of $(M_i,g_i,p_i)$ to $(N,h,q)$ to obtain appropriate smooth embeddings $\Phi_i$ of a superset of $B_R(q)$ into $(M_i,g_i)$ for $i$ sufficiently large, then restrict to $\Phi_i: B_{7r_0} \to M_i$.  Then $h_i:=\Phi_i^*g_i$ converges in $C^2$ to $h$ on $B_{7r_0}$.  Thus, in conjunction with (ii)-(vi) above, there exists $i_0 > 1$ so that for all $i \geq i_0$,
\begin{enumerate}
\item [(ii$'$)] Each $S_r$, for $r \in [r_0, 7r_0]$, has positive mean curvature with respect to every $h_i$.
\item [(iii$'$)] Distances in $B_{7r_0}$ with respect to $h$ and $h_i$ differ by a factor of at most $1+\eta_0$.
\item [(iv$'$)] Areas of surfaces in $B_{7r_0}$ with respect to $h$ and $h_i$ differ by a factor of at most $1+\eta_0$.
\item [(v$'$)] The sectional curvature of $h_i$ on $B_{7r_0} \sm B_{r_0}$ is bounded above by some
constant $\kappa_2$, where $e^{6\sqrt{\kappa_2} r_0} \leq (1 + \eta_0)^2$.
\item [(vi$'$)] Any point in $S_{4r_0}$ has injectivity radius (with respect to $h_i$) at least $\frac{3r_0}{(1+\eta_0)^2}$.
\end{enumerate}

With the aim of eventually applying Theorem \ref{thm_imcf}, we claim that for $i\geq i_0$, $\Sigma_i:=\Phi_i(S_{r_0})$ is area-outer-minimizing in $(M_i,g_i)$. Let $\tilde \Sigma_i$ be the outermost minimal area enclosure of $\Sigma_i$ in $(M_i,g_i)$. \footnote{That is, $\tilde \Sigma_i$ encloses $\Sigma_i$, has the least area among all surfaces enclosing $\Sigma_i$, and is the outermost such surface.  Existence of $\tilde \Sigma_i$ follows from asymptotic flatness and standard geometric measure theory arguments.  Standard regularity results imply that $\tilde \Sigma_i$ is a $C^{1,1}$ closed, embedded surface and that $\tilde \Sigma_i \sm \Sigma_i$, if nonempty, is a smooth minimal surface (cf. Theorem 1.3 of \cite{imcf}).}
\vspace{2mm}
\paragraph{\emph{Case 1:}} $\tilde \Sigma_i$ has a connected component disjoint from $\Sigma_i$.  Then $(M_i,g_i)$ contains a compact minimal surface, contrary to the hypothesis.
\vspace{2mm}
\paragraph{\emph{Case 2:}} $\tilde \Sigma_i$ is contained in $\Phi_i(B_{7r_0})$.  Since $\Phi_i:(B_{7r_0}, h_i) \to (\Phi_i(B_{7r_0}),g_i)$ is an isometry, we see $\Phi_i\inv(\tilde \Sigma_i)$ is a minimal surface in $(B_{7r_0},h_i)$. Moreover, $\Phi_i\inv(\tilde \Sigma_i)$ is contained in $B_{7r_0} \sm \interior(B_{r_0})$ because $\tilde \Sigma_i$ encloses $\Sigma_i$.
Let $r_{\text{max}} \in [r_0, 7r_0]$ denote the maximum value of the function $r$ restricted to 
$\Phi_i\inv(\tilde \Sigma_i)$.  If $r_{\text{max}} > r_0$, then $S_{r_{\text{max}}}$ encloses 
$\Phi_i\inv(\tilde \Sigma_i)$ and moreover these surfaces share a tangent plane at some point. 
 This contradicts the comparison principle for mean curvature: $S_{r_{\text{max}}}$ has positive mean curvature with respect to $h_i$ (by (ii$'$)) yet encloses the minimal surface $\Phi_i\inv(\tilde \Sigma_i)$ to which it is tangent.  Thus,
   $r_{\text{max}} = r_0$, which implies $\tilde \Sigma_i = \Sigma$.  It follows that $\Sigma_i$
   is area-outer-minimizing, which was claimed.
\vspace{2mm}
\paragraph{\emph{Case 3:}} If neither case 1 nor case 2 holds, then $\tilde \Sigma_i$ is connected and must intersect $\Sigma_i$ but contain a point outside of $\Phi_i(B_{7r_0})$.  By continuity there exists a point $a \in  \tilde \Sigma_i \cap \Phi_i(S_{4r_0})$.

Recall the monotonicity formula of Colding and Minicozzi (equation (5.5) of \cite{CM}) regarding minimal surfaces in a 3-manifold of bounded sectional curvature.  Suppose a point $x_0$ belongs to a smooth minimal surface $S$ embedded in a Riemannian 3-manifold.  Suppose the sectional curvatures of the 3-manifold are bounded in absolute value by a constant $k$, and the injectivity radius at $x_0$ is $i_0$.  Then for all $t \in \left(0, \min\left(i_0, \frac{1}{\sqrt{k}},\dist(x_0,\partial S\right)\right)$, the monotonicity formula states
$$\frac{d}{dt} \Theta(t) \geq 0, \qquad\qquad \text{where}\;\;\Theta(t)= \frac{e^{2\sqrt{k} t} |B_{t}(x_0) \cap S|}{\pi t^2},$$
where the area $|\cdot|$ and ball $B_{t}(x_0)$ are taken with respect to the Riemannian metric on the 3-manifold.  By the smoothness of $S$ it is clear that $\lim_{t \to 0^+} \Theta(t) = 1$, so that $\Theta(t) \geq 1$ for the allowable values of $t$.  We apply this formula to the minimal surface
$S= \Phi_i\inv(\tilde \Sigma_i \cap \Phi_i(B_{7r_0}))\sm S_{r_0}$  in $(B_{7r_0},h_i)$, with $x_0=\Phi_i\inv(a)$ and $k= \kappa_2$.

We first determine a large value of $t$ for which the monotonicity formula is valid.  First, note that the distance from $x_0$ to the boundary of $B_{7r_0} \sm B_{r_0}$ with respect to $h_i$
is bounded above by $\frac{1}{(1+\eta_0)^2}$ times this distance in the Euclidean metric, which is $3r_0$.  Here, we have used (iii) and (iii$'$).  By (vi$'$), the injectivity radius of $h_i$ at $x_0$ is at least $\frac{3r_0}{(1+\eta_0)^2}$.  Finally, by (v$'$), it follows that
$$\frac{1}{\sqrt{\kappa_2}} \geq \frac{3r_0}{\eta_0} \geq \frac{3r_0}{(1+\eta_0)^2}.$$
Thus, the monotonicity formula holds for $t \in \left(0, \frac{3r_0}{(1+\eta_0)^2}\right)$.  We choose $t = \frac{3r_0}{(1+\eta_0)^3}$.  Then
\begin{align*}
1 &\leq \Theta\left(\frac{3r_0}{(1+\eta_0)^3}\right)\\
&= \frac{e^{2\sqrt{\kappa_2} \frac{3r_0}{(1+\eta_0)^3}} |B_{t}(x_0) \cap S|_{h_i}}{\pi \left(\frac{3r_0}{(1+\eta_0)^3}\right)^2}\\
&\leq (1+\eta_0)^6\left(\frac{e^{6\sqrt{\kappa_2} r_0} |S|_{h_i}}{9\pi r_0^2 }\right)\\
&\leq (1+\eta_0)^8\left(\frac{|S|_{h_i}}{9\pi r_0^2 }\right),
\end{align*}
having used (v$'$) in the last step. Since $|\eta_0|<0.01$ by hypothesis,
$$|S|_{h_i} \geq \frac{9\pi r_0^2}{(1+\eta_0)^8} \geq 8.3 \pi r_0^2.$$

On the other hand, 
\begin{equation}
\label{ineq_chain}
|S|_{h_i} \leq |\tilde \Sigma_i|_{g_i} \leq |\Sigma_i|_{g_i}= |\Phi(S_{r_0})|_{g_i} = |S_{r_0}|_{h_i}.
\end{equation}
The first inequality holds because $S$ (with metric induced by $h_i$) is isometric to a subset to $\tilde \Sigma_i$ (with metric induced by $g_i$), and the second by the definition of $\tilde \Sigma_i$.   Finally, by (iv) and (iv'), the area of $S_{r_0}$ with respect to $h_i$ is bounded above by $(1+\eta_0)^2|S_{r_0}|_{\delta} = 4\pi (1+\eta_0)^2 r_0^2 \leq 4.1\pi r_0^2$. 
Together with (\ref{ineq_chain}), it follows that $|S|_{h_i} \leq 4.1\pi r_0^2$, a contradiction, so that case 3 does not occur.

\vspace{2mm}

Consideration of the above three cases establishes the claim that $\Sigma_i=\Phi_i(S_{r_0})$ is area-outer-minimizing in $(M_i,g_i)$ for $i \geq i_0$.  By Theorem \ref{thm_imcf} applied to the manifold-with-boundary obtained by removing the interior of $\Phi_i(B_{r_0})$ from $(M_i, g_i)$, we conclude
\begin{equation}
\label{estimate_mass}
m_{ADM}(M_i,g_i) \geq m_H^{(i)}(\Sigma_i)
\end{equation}
for $i \geq i_0$, where $m_H^{(i)}(\cdot)$ is the Hawking mass computed in $(M_i,g_i)$.

Completing the argument is now straightforward: by the $C^2$ convergence of $h_i$ to $h$, the Hawking mass of $S_{r_0}$ with respect to $h_i$ converges to $m_H(S_{r_0})$ as $i \to \infty$.  
Since $\Phi_i$ is an isometry, the former is equal to $m_H^{(i)}(\Sigma_i)$.  Thus, we may increase $i_0$ to ensure that
\begin{equation}
\label{estimate_hawking_masses}
|m_H^{(i)}(\Sigma_i)-m_H(S_{r_0})|< \frac{\epsilon}{2}
\end{equation}
for $i \geq i_0$.  Finally, for $i \geq i_0$,
\begin{align*}
m_{ADM}(N,h)  &\leq  m_H(S_{r_0})+ \frac{\epsilon}{2} &&\text{by (i)}\\
&\leq  m_H^{(i)}(\Sigma_i)+ \epsilon&&\text{by (\ref{estimate_hawking_masses})}\\
&\leq m_{ADM}(M_i,g_i)+ \epsilon && \text{by (\ref{estimate_mass})}
\end{align*}
Since $\epsilon$ was arbitrary, we may take $\liminf_{i \to \infty}$ to complete the proof.
\end{proof}

\section{Lower semicontinuity in rotational symmetry: weak convergence}
\label{sec_rot_sym}
Now we transition the discussion to rotationally symmetric manifolds and a natural notion of weak convergence. Subsections \ref{subsec_rot_sym}, \ref{subsec_embedding}, and \ref{subsec_weak} give the preliminaries for precisely stating and proving Theorem \ref{thm_main} in subsection \ref{subsec_proof}.

\subsection{Rotationally symmetric manifolds}
\label{subsec_rot_sym}
We consider rotationally symmetric, smooth Riemannian $n$-manifolds $(M,g)$, where $g$ is of the form
\begin{equation}
\label{eqn_metric_form}
g = ds^2 + h(s)^2 g_{S^{n-1}}.
\end{equation}
Here, $g_{S^{n-1}}$ is the standard metric on the unit $(n-1)$-sphere and $h:[0,\infty) \to [0,\infty)$ is a smooth function.  Note that $s$ is the distance from the boundary (if nonempty, so that $h(0)>0$ ) or the pole (if the boundary is empty, so that $h(0)=0$ and further $h'(0)=1$ by smoothness).  If $h(0)>0$, we require that $h'(0)=0$, which is equivalent to stating that the boundary sphere is a minimal hypersurface.  We further assume that $h'(s)>0$ for $s > 0$ and that $h(s)$ limits to infinity as $s \to \infty$. In geometric terms, these conditions mean that the $s=$ constant hyperspheres $S_s$ have positive mean curvature for $s>0$ and that their areas grow arbitrarily large as $s \to \infty$.  Finally, we assume that $g$ has nonnegative scalar curvature.  We denote by $\RotSym$ the class of Riemannian $n$-manifolds satisfying these conditions (borrowing notation from similar classes in \cite{lee_sormani} and \cite{lefloch_sormani}).

Recall the definition of the Hawking mass, formula (\ref{def_hawking}).  A well-known fact is the monotonicity of the Hawking mass: if $(M,g) \in \RotSym$, then the function $s \mapsto m_H(S_s)$ is monotone non-decreasing.  Thus, the following limit is well-defined (possibly $+\infty$):
\begin{equation}
\label{def_adm_limit}
m_{ADM}(M,g) = \lim_{s \to \infty} m_H(S_s).
\end{equation}
In the case that $(M,g)$ is asymptotically flat (see Definition \ref{def_AF}), the above limit agrees with the usual definition of the ADM mass (equation (\ref{eqn_ADM})).  Direct computation shows that
\begin{equation}
\label{eqn_hawking_metric}
m_H(S_s) = \frac{1}{2}h(s)^{n-2}\left(1-\left(\frac{dh}{ds}\right)^2\right).
\end{equation}
Since $\displaystyle\lim_{s \to 0^+} m_H(S_s) \geq 0$, and $m_H(\Sigma_s)$ is non-decreasing, we see $m_{ADM}(M,g) \geq 0$.

\subsection{Euclidean embedding}
\label{subsec_embedding}
A remarkable fact, exploited in \cite{lee_sormani} for example, is that any $(M,g) \in \RotSym$ may be realized isometrically as a graphical hypersurface in $\R^{n+1}$.  That is, there exists a subset $\Omega$ of $\R^n$ (equal to $\R^n$ if $M$ has no boundary and equal to $\R^n$ minus an open round ball about the origin if $M$ has a boundary) and a continuous function $f: \Omega \to \R$, smooth on the interior of $\Omega$, such that
$$\graph(f) = \{(\vec x, f(\vec x)) \in \R^{n+1} \; | \; \vec x \in \Omega\}$$
is isometric to $(M,g)$.   By symmetry, we may regard $f:[a,\infty) \to \R$ as a radial function $f(r)$, where $r$ denotes the Euclidean distance to the origin in $\R^n$.  Note that $a=h(0)$, the radius of the boundary sphere.  We call $f(r)$ a graphical representation of $(M,g)$; note that adding a constant to $f$ changes the embedding but not the induced metric.

An explicit formula for $f$ in terms of the metric of the form (\ref{eqn_metric_form}) is given as follows.  Since $h'(s)>0$, there exists an inverse function to $r=h(s)$, say $s=h\inv(r)$.  Since $\lim_{s \to 0^+} m_H(S_s) \geq 0$, and $m_H(S_s)$ is non-decreasing, (\ref{eqn_hawking_metric}) shows that $\left|\frac{dh}{ds}\right| \leq 1$.  Thus, the inverse function satisfies $\left|\frac{dh\inv(r)}{dr}\right| \geq 1$.  Direct computation shows that
\begin{equation}
\label{eqn_graph_function}
f(r)= \int^r_{h(0)} \left(\sqrt{\left(\frac{dh\inv(r)}{dr}\right)^2-1}\right)dr+K,
\end{equation}
where $K$ is any real constant.

We let $\Sigma_r$ denote the surface in $\graph(f)$ lying above the $r=$ constant coordinate sphere in $\R^n$.  Direct computation using (\ref{eqn_hawking_metric}) and (\ref{eqn_graph_function}) shows
\begin{equation}
\label{eqn_hawking_coordinate}
m_H(\Sigma_r) = \frac{1}{2}r^{n-2} \frac{f'(r)^2}{1+f'(r)^2}.
\end{equation}
We orient $\graph(f)$ as a hypersurface in $\R^{n+1}$ by choosing its normal to have positive dot product with $(0,\ldots,0,1)$.

\subsection{Weak convergence}
\label{subsec_weak}
We briefly recall weak convergence (in the sense of currents).  Let $\{S_i\}$ denote a sequence of oriented Lipschitz hypersurfaces in $\R^{n+1}$.  We may regard each such surface as a current, i.e. a functional on the space of compactly supported differential $n$-forms $\varphi$ in $\R^{n+1}$, by defining
$$S_i(\varphi) = \int_{S_i} \varphi.$$
We say $\{S_i\}$ converges \emph{weakly} to an oriented Lipschitz hypersurface $S$ if for all $\varphi$ as above, we have
$$\lim_{i \to \infty} S_i (\varphi) = S(\varphi).$$

\begin{remark}
In \cite{huang_lee}, Huang and Lee proved a stability result for the 
positive mass theorem, for the case of graphical hypersurfaces in $\R^{n+1}$ with respect to weak convergence.  
\end{remark}

\subsection{Lower semicontinuity in rotational symmetry for weak convergence}
\label{subsec_proof}

Below we prove the following theorem, which is a version of Theorem \ref{thm_intro} from the introduction that is stated more precisely.

\begin{thm}
\label{thm_main}
Let $\{f_i\}_{i=1}^\infty$ denote a sequence of graphical representatives for a sequence $\{(M_i,g_i)\}_{i=1}^\infty$ in $\RotSymm_n$.  Suppose $\{\graph(f_i)\}$ converges weakly to some nonempty Lipschitz hypersurface $N$ in $\R^{n+1}$, with induced metric $g$.  If the ADM mass of $N$ is defined, then
\begin{equation}
\label{eqn_lsc2}
m_{ADM}(N,h) \leq \liminf_{i \to \infty} m_{ADM}(M_i,g_i).
\end{equation}
\end{thm}
We emphasize that the additive constants of $f_i$ play a role in determining the limit space, in the same way that the choice of base points affects pointed convergence.  For instance, in example 6 in section \ref{sec_examples}, one can shift the spaces up and down in $\R^4$ to arrange that the limit space is a) a half-infinite cylinder, b) an infinite cylinder, c) a half-infinite cylinder attached to a Schwarzschild space, or d) empty.

Some of the delicate points in the proof include dealing with portions of the graphs running off to infinity, and the formation of cylindrical ends (as in example 6).  We outline the proof as follows:
\begin{itemize}
\item If the graph functions $f_i$ blow up at a fixed radius as $i \to \infty$, show that $N$ is contained inside a solid cylinder in $\R^{n+1}$.  Use the blowing up of the derivatives $f_i'$ near the cylinder boundary to establish (\ref{eqn_lsc2}).
\item Otherwise, show there is some radius $a_0$ beyond which all graph functions $f_i$ converge uniformly to some limit, $f$.
\item Argue that by virtue of the nonnegativity of scalar curvatures, $f_i'$ converges to $f'$ almost everywhere.  Also show $\graph(f)$ is contained in $N$, and that the ADM mass of $N$ is defined.
\item Use the monotonicity of the Hawking masses to establish lower semicontinuity in the separate cases in which the limit has either finite or infinite ADM mass.
\end{itemize}

\begin{proof}
Let $m_i \geq 0$ denote the ADM mass of $(M_i,g_i)$. 
First, pass to a subsequence of $\{(M_i, g_i)\}$ (denoted the same) for which the ADM masses limit to the $\liminf$ of the ADM masses of the original sequence.  
This allows us to pass to further subsequences without loss of generality.

\vspace{2mm}
\paragraph{\emph{Part 0:}}
Clearly we may assume $\displaystyle\liminf_{i \to \infty} m_i$ is finite (or else (\ref{eqn_lsc2}) is trivial), so that there exists an upper bound $C>0$ of $\{m_i\}$.  By the monotonicity of the Hawking mass (\ref{eqn_hawking_metric}) in each $(M_i,g_i)$ we have
\begin{equation}
\label{eqn_penrose}
C \geq m_i \geq m_H(S_0) =\frac{1}{2} h_i(0)^{n-2},
\end{equation}
where $h_i$ is the profile function (\ref{eqn_metric_form}) associated to $(M_i,g_i)$.  Thus $h_i(0) \leq (2C)^{\frac{1}{n-2}}< (4C)^{\frac{1}{n-2}}=:a_0$ is bounded above, independently of $i$.
It follows that there exists a single interval $[a_0, \infty)$ on which all $f_i$ are defined, for $i$ sufficiently large.  From here on, we assume $i$ is such.

\paragraph{\emph{Part 1:}}
We first consider the case in which there exists a number $r\geq a_0$ for which $\liminf_{i \to \infty} f_i(r) = +\infty$.  Fix $r_*\geq a_0$ as the infimum of such values of $r$.   
Since each $f_i$ is non-decreasing, we have $\liminf_{i \to \infty} f_i(r) = +\infty$ for all $r>r_*$.
Let $\Omega$ be the solid, closed cylinder of radius $r_*$ in $\R^{n+1}$ about the $x_{n+1}$ axis.
From the definition of weak convergence, it is clear $N \subset \Omega$.

If $r_* = 0$, then $N$ is not a Lipschitz hypersurface, a contradiction. Thus, assume $r_* > 0$. If the ADM mass of $(N,h)$ is defined, then it is bounded above by $\frac{1}{2}r_*^{n-2}$ by virtue of formulas (\ref{def_adm_limit}) and (\ref{eqn_hawking_metric}).  Now fix any number $c>1$.  By the definition of $r_*$, we have $\liminf_{i \to \infty} f_i(r_*/c) < +\infty$ and $\liminf_{i \to \infty} f_i(cr_*) = +\infty$.  Pass to a subsequence (again, using the same notation) for which $\{f_i(r_*/c)\}$ is uniformly bounded above.  By the mean value theorem, there exists $c_i \in [r_*/c, cr_*]$ for which $\lim_{i \to \infty} f_i'(c_i) = + \infty$.  

Let $m_H^{(i)}(\cdot)$ denote the Hawking mass with respect to $(M_i,g_i)$; recall that on any rotationally symmetric hypersphere, $m_H^{(i)}$ provides a lower bound of $m_{ADM}(M_i,g_i)$. Then
\begin{align*}
m_{ADM}(N,h) &\leq \frac{1}{2} r_*^{n-2}\\
&= \lim_{i \to \infty} \frac{1}{2} r_*^{n-2} \frac{f_i'(c_i)^2}{1+f_i'(c_i)^2}\\
&\leq c^{n-2} \liminf_{i \to \infty} \frac{1}{2} c_i^{n-2} \frac{f_i'(c_i)^2}{1+f_i'(c_i)^2}\\
&= c^{n-2} \liminf_{i \to \infty} m_H^{(i)} (\Sigma_{c_i})\\
&\leq c^{n-2} \liminf_{i \to \infty} m_{ADM} (M_i,g_i).
\end{align*}
Since $c>1$ was arbitrary, the proof is complete for this case.

\vspace{2mm}
\paragraph{\emph{Part 2:}}For the rest of this proof, we may now assume
\begin{equation}
\label{liminf_finite}
\liminf_{i \to \infty} f_i(r) < + \infty \text{ for every } r \geq a_0.
\end{equation}

By the monotonicity of the Hawking mass (\ref{eqn_hawking_coordinate}) in $(M_i,g_i)$, we have for each $r\geq a_0$,
$$\frac{r^{n-2}f_i'(r)^2}{2(1+f_i'(r)^2)} \leq m_i \leq C.$$
Since $a_0=(4C)^{\frac{1}{n-2}}$,  we have
\begin{equation}
\label{eqn_f_i_prime_bound}
0 \leq f_i'(r) \leq \sqrt{\frac{2C}{r^{n-2}-2C}} \leq 1,
\end{equation}
for $r \geq a_0$.

\begin{lemma}
\label{lemma_uniform_convergence}
The sequence $\{f_i\}$ of functions restricted to $[a_0, \infty)$ converges uniformly on compact sets to a continuous function $f:[a_0,\infty) \to \R$.
\end{lemma}
\begin{proof}
We first show pointwise convergence of $\{f_i\}$.  Fix $a \geq a_0$.  We analyze the sequence $\{f_i(a)\}$.

By (\ref{liminf_finite}), $L:=\liminf_{i \to \infty} f_i(a) < +\infty$.  
If $L = -\infty$, then since $f_i$ is nondecreasing for all values of $r$ and $f_i'(r) \leq 1$ for all $r \geq a$, the graph of $f_i$ eventually leaves any compact set as $i \to \infty$.  This would imply that the weak limit of $\{\graph(f_i)\}$ is the empty set (i.e., zero current), a contradiction.  Thus $L$ is finite.

Suppose $\{f_i(a)\}$ diverges.  Since $L$ is finite, there exist constants $z_0 \in \R$ and $\delta > 0$ and subsequences
$\{f_{i_k}(a)\}^{\infty}_{k=1}$ and $\{f_{j_k}(a)\}^{\infty}_{k=1}$ such that
$$L-1\leq f_{i_k}(a) \leq z_0 -2\delta, \text{ and }  \qquad f_{j_k}(a) \geq z_0 + 2\delta$$
for all $k \geq 1$.   By (\ref{eqn_f_i_prime_bound}), for $r$ in the interval $[a,a+\delta]$,
$$L-1 \leq f_{i_k}(r) \leq z_0 -\delta,  \text{ and } \qquad f_{j_k}(r) \geq z_0 + 2\delta.$$
Let $\varphi$ be the differential $n$-form on $\R^{n+1}$ given by $\rho dx_1 \wedge \ldots \wedge dx_n,$ where $\rho \geq 0$ is a smooth function of compact support with $\rho(\vec x,x_{n+1})=0$ if $x_{n+1} \geq z_0$ and $\rho(\vec x, x_{n+1})=1$ on the solid truncated annular cylinder given by 
$$ a \leq |\vec x| \leq a+\delta, \qquad L-1\leq x_{n+1} \leq z_0 -\delta.$$  
Then $\graph(f_{i_k})(\varphi)$ equals the $n$-volume of $ a \leq |\vec x| \leq a+\delta$ in $\R^n$ for each $k\geq 1$, while $\graph(f_{j_k})(\varphi)=0$ for each $k\geq 1$.  This contradicts the assumption on weak convergence.

It follows that $\{f_i\}$ converges pointwise to some function $f:[a_0,\infty) \to \R$.  By (\ref{eqn_f_i_prime_bound}) the convergence is uniform on compact sets, and so $f$ is continuous.
\end{proof}

\vspace{2mm}
\paragraph{\emph{Part 3:}}
Note that by Lemma \ref{lemma_uniform_convergence} and (\ref{eqn_f_i_prime_bound}), $f$ is Lipschitz, with Lipschitz constant at most 1.   By Rademacher's theorem,
$f$ is differentiable almost everywhere.  The next lemma establishes that the derivatives converge almost everywhere (which is false without assuming
the graphs of $f_i$ have nonnegative scalar curvature).

\begin{lemma}
\label{lemma_derivs_converge}
If $f'(r_0)$ is defined, then $\displaystyle \lim_{i \to \infty}f_i'(r_0) = f'(r_0)$.
\end{lemma}
The proof of the lemma appears following the conclusion of the proof of Theorem \ref{thm_main}.

Now we establish weak convergence of the graphs outside a compact set.  Let $\graph_{a,b}(f)$ denote the graph of $f$ in $\R^{n+1}$, restricted between radii $a$ and $b$ in $[a_0,\infty]$.

\begin{lemma}
The set $\graph(f)$ is a Lipschitz hypersurface in $\R^{n+1}$, and the sequence $\{\graph_{a_0,\infty}(f_i)\}$ converges weakly to $\graph(f)$ as $i \to \infty$.
\end{lemma}

\begin{proof}
Since $f$ is a Lipschitz function, its graph is a Lipschitz hypersurface.  By uniform convergence on compact sets, it is clear that for every $b > a_0$, $\graph_{a_0,b}(f_i)$ converges in the flat norm to $\graph_{a_0,b}(f)$.  To see this, let $A_i$ denote the $(n+1)$-current defined by the region between $\graph_{a_0,b}(f_i)$ and $\graph_{a_0,b}(f)$, oriented appropriately.  Let $B_i$ denote the $n$-current defined by the cylinders between the graphs of $f_i$ and $f$ at radii $a_0$ and $b$, oriented appropriately.  Then 
$$\graph_{a_0,b}(f_i) -\graph_{a_0,b}(f) = \partial A_i + B_i.$$
Moreover, the $(n+1)$-volume of $A_i$ and the $n$-volume of $B_i$ both converge to $0$ by uniform convergence of $f_i$ to $f$ on compact sets.  Thus,  we have convergence in the flat norm.

That weak convergence follows is well-known: given any compactly supported differential $n$-form $\varphi$ on $\R^{n+1}$, there exists a constant $k>0$ such that $\|\varphi\| \leq k$ and $\|d\varphi\| \leq k$ pointwise.  Choose $b>0$ sufficiently large so that the support of $\varphi$ is contained in the cylinder in $\R^{n+1}$ of radius $b$ about the $x_{n+1}$ axis.
Then
\begin{align*}
|\graph_{a_0,\infty}(f_i)(\varphi) -\graph_{a_0,\infty}(f)(\varphi)| &=
|\graph_{a_0,b}(f_i)(\varphi) -\graph_{a_0,b}(f)(\varphi)|\\
 &= |\partial A_i(\varphi) + B_i(\varphi)|\\
&\leq |A_i(d\varphi)| + |B_i(\varphi)|\\
&\leq k \vol_{n+1}(A_i) + k \vol_n(B_i),
\end{align*}
which converges to zero.  Here, $\vol_{p}$ denotes the $p$-dimensional volume (which is traditionally called the ``mass'' of a current, a term we avoid here).
\end{proof}

Since $\graph(f_i)$ is assumed to converge weakly to some Lipschitz hypersurface $N$, it is clear that $\graph(f)$ and $N$ are equal when intersected with the  complement of the closed cylinder of radius $a_0$ about the $x_{n+1}$ axis.  Thus, we define the ADM mass
of $N$ as the ADM mass of $\graph(f)$, provided the latter is well-defined.  The following argument shows this to be the case.

Note that the spheres $\Sigma_r$ in $\graph(f)$, for $r \geq a_0$, have Hawking mass defined for almost all values of $r$, by formula (\ref{eqn_hawking_coordinate}) and the fact that $f$ is differentiable almost everywhere.  Now, Lemma
\ref{lemma_derivs_converge} shows that the Hawking mass functions for $\graph(f_i)$ converge pointwise  almost-everywhere to the Hawking mass function for $\graph(f)$ on $[a_0,\infty)$ as $i \to \infty$.  Since the former are non-decreasing for each $i$, it follows
that the Hawking mass function for $\graph(f)$ is non-decreasing as well.
In particular, the ADM mass of $\graph(f)$ (and hence $N$) is well-defined as the limit $r \to \infty$ of the Hawking mass, possibly $+\infty$.

\vspace{2mm}
\paragraph{\emph{Part 4:}}
For $r \geq a_0$, let $m_H^{(i)}(\Sigma_r)$ and $m_H(\Sigma_r)$ denote the Hawking masses of $\Sigma_r$ computed respectively in $\graph(f_i)$ or $\graph(f)$ (if defined in the latter case).  Let $m$ denote the ADM mass of $\graph(f)$ (which equals the ADM mass of $N$).

First, consider the case in which $m$ is finite. Let $\epsilon > 0$.  
Since $\lim_{r \to \infty} m_H(\Sigma_r) = m$, there exists $a_1 \geq a_0$ sufficiently large so that
$$|m - m_H(\Sigma_r)| < \frac{\epsilon}{2}$$
for all $r \geq a_1$ for which $m_H(\Sigma_r)$ is defined.  By Lemma \ref{lemma_derivs_converge}, there exists $a_2 \geq a_1$ for which $f_i'(a_2)$ converges $f'(a_2)$ as $i \to \infty$, so that $m_H^{(i)}(\Sigma_{a_2})$ converges to $m_H(\Sigma_{a_2})$ by (\ref{eqn_hawking_coordinate}). Consequently, there exists $i_0 \geq 1$ such that
$$|m_H^{(i)}(\Sigma_{a_2}) - m_H(\Sigma_{a_2})| < \frac{\epsilon}{2}$$
for all $i \geq i_0$.  In particular, for $i \geq i_0$,
%$$m < \epsilon + m_H^{(i)}(\Sigma_{a_2}) \leq \epsilon + m_i,$$
$$m < m_H(\Sigma_{a_2}) + \frac{\epsilon}{2} <  m_H^{(i)}(\Sigma_{a_2}) + \epsilon \leq  m_i +\epsilon,$$
where we have used the fact that the Hawking mass in $(M_i,g_i)$ limits monotonically to the ADM mass thereof.  Taking $\liminf_{i \to \infty}$ completes the proof for the case of $m$ finite, as $\epsilon$ was arbitrary.

Second, suppose $m=+\infty$, and let $C_0>0$ be any large constant.  There exists $a_1\geq a_0$ sufficiently large so that 
$$m_H(\Sigma_r) \geq 2C_0$$
for all $r \geq a_1$ for which $m_H(\Sigma_r)$ is defined.  For some $a_2 \geq a_1$, Lemma \ref{lemma_derivs_converge} and (\ref{eqn_hawking_coordinate}) assure that
$m_H^{(i)}(\Sigma_{a_2})$ converges to $m_H(\Sigma_{a_2})$, so that for some $i_0 \geq 1$,
$$|m_H^{(i)}(\Sigma_{a_2}) - m_H(\Sigma_{a_2})| < C_0,$$
for $i \geq i_0$. Then by the monotonicity of the Hawking mass,
%$$m_i \geq m_H^{(i)}(\Sigma_{a_2}) \geq m_H(\Sigma_{a_2}) - C_0 \geq C_0.$$
$$ C_0 \leq m_H(\Sigma_{a_2}) - C_0 \leq m_H^{(i)}(\Sigma_{a_2}) \leq m_i.$$

Thus $\liminf_{i \to \infty} m_i \geq C_0$, where $C_0$ is arbitrary.  The proof is complete.
\end{proof}

Now we return to the proof of the key lemma giving almost-everywhere convergence of the derivatives of the graph functions:
\begin{proof}[Proof of Lemma \ref{lemma_derivs_converge}]
Assume $f$ is differentiable at $r_0$. If $\{f_i'(r_0)\}$ does not converge to $f'(r_0)$, there exists $\epsilon > 0$ and a subsequence (of the same name, say) for which either
\begin{equation}
\label{case1}
f_i'(r_0) \geq f'(r_0) + 3\epsilon
\end{equation}
or
\begin{equation}
\label{case2}
f_i'(r_0) \leq f'(r_0) - 3\epsilon.
\end{equation}
Assume first that (\ref{case1}) holds.  Since the graph of $f_i$ has nonnegative scalar curvature, the Hawking mass function (\ref{eqn_hawking_coordinate}) 
$$r^{n-2} \frac{f_i'(r)^2}{1+f_i'(r)^2}$$
is non-decreasing as a function of $r$ for each $i$.  Let $H:[0,\infty) \to [0,1)$ be the increasing homeomorphism $H(y) = \frac{y^2}{1+y^2}$.  Then for $r \geq r_0$, we have
$$r^{n-2} H(f_i'(r)) \geq r_0^{n-2} H(f_i'(r_0)) \geq r_0^{n-2} H(f'(r_0) + 3\epsilon)$$
for each $i$ by (\ref{case1}). Since $H\inv$ is increasing, we have
$$f_i'(r) \geq H\inv \left(\frac{r_0^{n-2}}{r^{n-2}} H\big(f'(r_0) + 3\epsilon\big)\right)$$
for each $i$.  The right-hand side defines a continuous function of $r$ on $[r_0,\infty)$ that limits to $f'(r_0) + 3\epsilon$ as $r\searrow r_0$.  Thus, there exists $\delta > 0$ such that
$$f_i'(r) \geq f'(r_0) + 2\epsilon$$
for $r \in [r_0,r_0 + \delta]$ and all $i$.  Since $f$ is differentiable at $r_0$, we may shrink $\delta$ if necessary, independently of $i$, to arrange that
$$f_i'(r) \geq \frac{f(r)-f(r_0)}{r-r_0} + \epsilon$$
for $r \in [r_0,r_0 + \delta]$ and all $i$.  Both the left- and right-hand sides are continuous functions on $[r_0,r_0+\delta]$, so that we may integrate from $r_0$ to $r_0+c$, where $c \in(0,\delta]$, to obtain
$$f_i(r_0+c)-f_i(r_0) \geq  \int_{r_0}^{r_0+c} \frac{f(r)-f(r_0)}{r-r_0} dr + c\epsilon$$
for each $i$.
Take the limit as $i \to \infty$, using the pointwise convergence of $f_i$ to $f$:
$$\frac{f(r_0+c)-f(r_0)}{c} \geq  \frac{1}{c}\int_{r_0}^{r_0+c} \frac{f(r)-f(r_0)}{r-r_0} dr + \epsilon.$$
Taking the limit $c \to 0^+$ implies $f'(r_0) \geq f'(r_0) + \epsilon$, a contradiction.

The proof of case (\ref{case2}) is very similar.
\end{proof}

\section{Lower semicontinuity in rotational symmetry: intrinsic flat convergence}
\label{sec_IF}

%\section{Discussion: intrinsic flat convergence}
%\label{sec_discuss}
A natural extension of the present work is to consider the question of lower semicontinuity of the ADM mass for other modes of convergence, such as pointed convergence in the intrinsic flat distance of Sormani and Wenger \cite{sormani_wenger}.

The intrinsic flat distance shows promise for applications to general relativity.  We refer specifically to results of Lee and Sormani on the stability of the positive mass theorem and Penrose inequality for the intrinsic flat distance within a class of rotationally symmetric manifolds \cites{lee_sormani,lee_sormani2}, and to a compactness result of LeFloch and Sormani within the same class (allowing lower regularity) \cite{lefloch_sormani}.  The very basic idea behind defining the intrinsic flat distance between compact Riemannian manifolds is to 1) embed them isometrically into a complete metric space, 2) view their images as generalized integral currents and compute their flat distance, then 3) minimize over all such metric spaces and isometric embeddings.

Below we state and prove a lower semicontinuity result for the ADM mass on the class $\RotSym$ (defined in section \ref{subsec_rot_sym}) with respect to a pointed notion of intrinsic flat convergence.  

To set some notation, suppose $(M,g) \in \RotSym$.  For each number $A> |\partial M|$, there exists a unique rotationally symmetric sphere $\Sigma^A_g$ in $M$ with area $A$.  Define $U_g^A$ to be the open region bounded by $\Sigma^A_g$ and $\partial M$.  Define $U_g^A$ to be empty if $A \leq |\partial M|$.
\begin{thm}
Let $\{(M_i,g_i)\}$ denote a sequence in $\RotSymm_n$ and let $(N,h) \in \RotSymm_n$. 
Assume $\liminf_{i \to \infty} m_{ADM}(M_i,g_i)$ is finite.\footnote{If the $\liminf$ is infinite, (\ref{eqn_LSC_IF})  follows trivially.}  Then for each $A>0$ sufficiently large, $(U_{g_i}^A, g_i)$ is nonempty for $i$ sufficiently large.

Assume that for almost every $A>0$ sufficiently large, $(U_{g_i}^A, g_i)$ converges in the intrinsic flat distance to $(U_h^A,h)$, and that the diameter of $(U_{g_i}^A, g_i)$ is bounded above independently of $i$. Then
\begin{equation}
\label{eqn_LSC_IF}
m_{ADM}(N,h) \leq \liminf_{i \to \infty} m_{ADM}(M_i,g_i).
\end{equation}
\end{thm}
The main ingredient in the proof is a compactness theorem of LeFloch and Sormani \cite{lefloch_sormani}.

\begin{proof}
 Passing to a subsequence (denoted the same), there exists a uniform upper bound $C> 0$ of $\{m_{ADM}(M_i,g_i)\}$.  By (\ref{eqn_penrose}), there also exists a uniform upper bound for the boundary areas $|\partial M_i|_{g_i}$, which shows that for $A$ large enough,  
 $(U_{g_i}^A, g_i)$ is eventually nonempty as $i \to \infty$.

Let $\epsilon > 0$.  Since the Hawking mass of rotationally symmetric spheres monotonically increases to the ADM mass, there exists $A>0$ such that
\begin{equation}
\label{eqn_masses_IF}
m_{ADM}(N,h) - m_H(\Sigma^A) < \frac{\epsilon}{2},
\end{equation}
where $\Sigma^A = \partial U_h^A \sm \partial N$ and $m_H$ is the Hawking mass with respect to $h$.  If necessary, increase $A$ so that the hypotheses of the theorem apply: 
$(U_{g_i}^A, g_i)$ converges in the intrinsic flat distance to $(U_h^A,h)$, and the diameter of $(U_{g_i}^A, g_i)$ is bounded above independently of $i$.  The latter shows the \emph{depth} of $\partial U_{g_i}^A$ to be bounded above independently of $i$, as defined in \cite{lefloch_sormani}.

By Theorem 8.1 of \cite{lefloch_sormani}\footnote{Note that our class $\RotSym$ is a subset of the class $\overline{\RotSym}^{\text{weak},1}$ considered in \cite{lefloch_sormani}.}, a subsequence of $(U_{g_i}^A, g_i)$ converges in the intrinsic flat distance to some limit; by our hypothesis, the limit is $(U_h^A,h)$.  This theorem also establishes that the Hawking masses of $\partial U^A_{g_i} \sm \partial M_i$ converge to $m_H(\Sigma^A)$ as $i \to \infty$.  

Putting this all together, the ADM mass of $(M_i,g_i)$ is at least the Hawking mass of $\partial U_{g_i}^A \sm \partial M_i$ with respect to $g_i$, which is within $\frac{\epsilon}{2}$ of $m_H(\Sigma^A)$ for $i$ sufficiently large.  By (\ref{eqn_masses_IF}), we see that $m_{ADM}(M_i,g_i)$ is at least $m_{ADM}(N,h) -\epsilon$ for $i$ sufficiently large.  Inequality (\ref{eqn_LSC_IF}) follows.
\end{proof}

We conjecture that (\ref{eqn_LSC_IF}) holds on the space of asymptotically flat $n$-manifolds of nonnegative scalar curvature containing no compact minimal surfaces, with respect to pointed intrinsic flat convergence.


\newpage

\begin{bibdiv}
 \begin{biblist}

 \bib{adm}{article}{
   author={Arnowitt, R.},
   author={Deser, S.},
   author={Misner, C.},
   title={Coordinate invariance and energy expressions in general relativity},
   journal={Phys. Rev. (2)},
   volume={122},
   date={1961},
   pages={997--1006},
}

\bib{bartnik_adm}{article}{
   author={Bartnik, R.},
   title={The mass of an asymptotically flat manifold},
   journal={Comm. Pure Appl. Math.},
   volume={39},
   date={1986},
   number={5},
   pages={661--693},
}

\bib{bartnik}{article}{
   author={Bartnik, R.},
   title={Energy in general relativity},
   conference={
      title={Tsing Hua lectures on geometry \& analysis},
      address={Hsinchu},
      date={1990--1991},
   },
   book={
      publisher={Int. Press, Cambridge, MA},
   },
   date={1997},
   pages={5--27},
}


\bib{bartnik2}{article}{
   author={Bartnik, R.},
   title={Mass and 3-metrics of non-negative scalar curvature},
   conference={
      title={Proceedings of the International Congress of Mathematicians,
      Vol. II },
      address={Beijing},
      date={2002},
   },
   book={
      publisher={Higher Ed. Press, Beijing},
   },
   date={2002},
   pages={231--240},
}


\bib{bray_finster}{article}{
   author={Bray, H.},
   author={Finster, F.},
   title={Curvature estimates and the positive mass theorem},
   journal={Comm. Anal. Geom.},
   volume={10},
   date={2002},
   number={2},
   pages={291--306},
}


\bib{CM}{book}{
   author={Colding, T.},
   author={Minicozzi, W., II},
   title={Minimal surfaces},
   series={Courant Lecture Notes in Mathematics},
   volume={4},
   publisher={New York University, Courant Institute of Mathematical
   Sciences, New York},
   date={1999},
}


\bib{dai_ma}{article}{
   author={Dai, X.},
   author={Ma, L.},
   title={Mass under the Ricci flow},
   journal={Comm. Math. Phys.},
   volume={274},
   date={2007},
   number={1},
   pages={65--80},
}

\bib{FST}{article}{
   author={Fan, X.-Q.},
   author={Shi, Y.},
   author={Tam, L.-F.},
   title={Large-sphere and small-sphere limits of the Brown-York mass},
   journal={Comm. Anal. Geom.},
   volume={17},
   date={2009},
   number={1},
   pages={37--72},
   }

\bib{finster_kath}{article}{
   author={Finster, F.},
   author={Kath, I.},
   title={Curvature estimates in asymptotically flat manifolds of positive
   scalar curvature},
   journal={Comm. Anal. Geom.},
   volume={10},
   date={2002},
   number={5},
   pages={1017--1031},
   issn={1019-8385},
}

\bib{haslhofer}{article}{
   author={Haslhofer, R.},
   title={A mass-decreasing flow in dimension three},
   journal={Math. Res. Lett.},
   volume={19},
   date={2012},
   number={4},
   pages={927--938},
}


\bib{huang_lee}{article}{
	author={Huang, L.-H.},
	author={Lee, D.},
	title={Stability of the positive mass theorem for graphical hypersurfaces of Euclidean space},
	date={2014},
	eprint={http://arxiv.org/abs/1405.0640},
}


\bib{huang_lee_sormani}{article}{
	author={Huang, L.-H.},
	author={Lee, D.},
	author={Sormani, S.},
	title={Intrinsic flat stability of the positive mass theorem for graphical hypersurfaces of Euclidean space},
	date={2014},
	eprint={http://arxiv.org/abs/1408.4319},
}


\bib{imcf}{article}{
   author={Huisken, G.},
   author={Ilmanen, T.},
   title={The inverse mean curvature flow and the Riemannian Penrose
   inequality},
   journal={J. Differential Geom.},
   volume={59},
   date={2001},
   number={3},
   pages={353--437},
}

\bib{lee}{article}{
   author={Lee, D.},
   title={On the near-equality case of the positive mass theorem},
   journal={Duke Math. J.},
   volume={148},
   date={2009},
   number={1},
   pages={63--80},
}

\bib{lee_sormani}{article}{
	author={Lee, D.},
	author={Sormani, C.},
	title={Stability of the positive mass theorem for rotationally symmetric Riemannian manifolds},
	journal={J. Reine Angew. Math.},
    volume={686},
    date={2014},
   pages={187--220},
}

\bib{lee_sormani2}{article}{
   author={Lee, D.},
   author={Sormani, C.},
   title={Near-equality of the Penrose inequality for rotationally symmetric
   Riemannian manifolds},
 %  journal={Ann. Henri Poincare},
   journal={Ann. Henri Poincar\'e},
   volume={13},
   date={2012},
   number={7},
   pages={1537--1556},
}


\bib{lefloch_sormani}{article}{
	author={LeFloch, P.},
	author={Sormani, C.},
	title={The nonlinear stability of rotationally symmetric spaces with low regularity},
	date={2014},
	eprint={http://arxiv.org/abs/1401.6192},
}

\bib{list}{article}{
   author={List, B.},
   title={Evolution of an extended Ricci flow system},
   journal={Comm. Anal. Geom.},
   volume={16},
   date={2008},
   number={5},
   pages={1007--1048},
}

\bib{OW}{article}{
   author={Oliynyk, T.},
   author={Woolgar, E.},
   title={Rotationally symmetric Ricci flow on asymptotically flat
   manifolds},
   journal={Comm. Anal. Geom.},
   volume={15},
   date={2007},
   number={3},
   pages={535--568},
}

\bib{schoen_yau}{article}{
	author={Schoen, R.},
	author={Yau, S.-T.},
	title={On the proof of the positive mass conjecture in general relativity},
	journal={Comm. Math. Phys.},
	volume={65},
	year={1979},
	pages={45--76},
}


\bib{sormani_wenger}{article}{
   author={Sormani, C.},
   author={Wenger, S.},
   title={The intrinsic flat distance between Riemannian manifolds and other
   integral current spaces},
   journal={J. Differential Geom.},
   volume={87},
   date={2011},
   number={1},
   pages={117--199},
}

\bib{witten}{article}{
	author={Witten, E.},
	title={A new proof of the positive energy theorem},
	journal={Comm. Math. Phys.},
	volume={80},
	year={1981},
	pages={381-402},
}	


\end{biblist}
\end{bibdiv}
\end{document}